\documentclass[a4paper,11pt]{amsart}
\usepackage[colorlinks,hyperindex,linkcolor=blue,urlcolor=black,
pdftitle={Explicit traces of functions on Sobolev spaces and quasi-optimal
linear interpolators},
pdfauthor={Daniel Estevez}
]{hyperref}
\usepackage{amsmath}
\usepackage{amsfonts}
\usepackage{amssymb}
\usepackage{amsthm}
\usepackage[english]{babel}
\usepackage{enumerate}
\usepackage{graphicx}

\theoremstyle{plain}
\newtheorem{theorem}{Theorem}
\newtheorem{lemma}[theorem]{Lemma}

\newtheorem{proposition}[theorem]{Proposition}
\newtheorem{theoremcite}{Theorem}
\newtheorem{lemmacite}[theoremcite]{Lemma}

\theoremstyle{definition}
\newtheorem{problem}{Problem}
\newtheorem{question}{Question}
\newtheorem{example}{Example}

\theoremstyle{remark}

\theoremstyle{definition}

\begin{document}

\title[\tiny{Explicit traces of functions on Sobolev spaces and
quasi-optimal linear interpolators}]
{Explicit traces of functions on Sobolev spaces and
quasi-optimal linear interpolators}

\author{Daniel Est\'evez}
\address{Departamento de Matem\'{a}ticas\\
Universidad Aut\'onoma de Madrid\\ Cantoblanco 28049 (Madrid)\\ Spain}
\email{daniel.estevez@uam.es}
\thanks{
The author has been supported by a grant of the Mathematics Department of the
Universidad Aut\'onoma de Madrid and the ICMAT-Intro grant of the Institute for
Mathematical Sciences, Spain.
}

\date\today

\begin{abstract}
Let $\Lambda \subset \mathbb{R}$ be a strictly increasing sequence. For $r =
1,2$, we give a simple explicit expression for an equivalent norm on the trace
spaces $W_p^r(\mathbb{R})|_\Lambda$, $L_p^r(\mathbb{R})|_\Lambda$ of the
non-homogeneous and homogeneous Sobolev spaces with $r$
derivatives $W_p^r(\mathbb{R})$, $L_p^r(\mathbb{R})$. We also construct an
interpolating spline of low degree having optimal norm up to a
constant factor. A general result relating interpolation in $L^r_p(\mathbb{R})$
and $W^r_p(\mathbb{R})$ for all $r \geq 1$ is also given.
\end{abstract}

\maketitle

\section{Introduction}

Let $I \subset \mathbb{R}$ be an interval (possibly infinite). The Sobolev
space $W_p^r(I)$ for integer $r \geq 1$ is defined as the completion of the
space of complex functions $F \in
\mathcal{C}^\infty(I)$ such that the norm
\[
 \|F\|_{W_p^r(I)}^p = \int_I |F(x)|^p\,dx + \int_I |F^{(r)}(x)|^p\,dx
\]
is finite. By the Sobolev embedding theorem, if $F \in W_p^r$, $1 \leq p < \infty$,
then $F^{(r-1)}$ is absolutely continuous.
Similarly, one can define the homogeneous Sobolev space $L_p^r(I)$ as the space
of functions of finite seminorm
\[
 \|F\|_{L_p^r(I)}^p = \int_I |F^{(r)}(x)|^p\,dx.
\]
In fact, given any $x_0\in I$, the expression
\[
\left(|F(x_0)|^p + \cdots + |F^{(r-1)}(x_0)|^p +
\|F\|^p_{L_p^r(I)}\right)^{\frac{1}{p}}
\]
gives a Banach-space norm on $L_p^r(I)$. If $I$ is a finite interval, then
$L_p^r(I) = W_p^r(I)$, and this Banach norm is equivalent to
$\|\cdot\|_{W_p^r(I)}$.

In the sequel, let $\mathbb{X}$ stand either for $W_p^r(I)$ or $L_p^r(I)$.
If $\Lambda \subset I$ is any set, we can define the trace space
\[
\mathbb{X}|_\Lambda = \{F|_\Lambda : F \in \mathbb{X}\}.
\]
There is the following natural choice of a norm (or seminorm) in the trace
space:
\[
\|f\|_{\mathbb{X}|_\Lambda} = \inf \{\|F\|_{\mathbb{X}} : F|_\Lambda = f\}.
\]

Throughout this note we
will restrict ourselves to the case where $\Lambda =
\{\lambda_n\}_{n\in\mathbb{Z}}$ is an increasing sequence:
\[
 \lambda_n < \lambda_{n+1},\qquad n\in\mathbb{Z}.
\]

As is known \cite{Fefferman05Sharp,Israel}, the following questions are of
relevance in the study of the trace space.

\begin{problem}
\label{problem1}
Given some $f:\Lambda \to \mathbb{C}$, when does $f$ extend to a function $F \in
\mathbb{X}$ satisfying $F|_\Lambda = f$? Can one take this extension to depend
linearly on the data $f$?
\end{problem}

\begin{problem}
\label{problem2}
Find an explicit simple formula for a norm (or seminorm) equivalent to
$\|\cdot\|_{\mathbb{X}|_\Lambda}$.
\end{problem}

Observe that in our case, $\mathbb{X}|_\Lambda$ is a sequence space. Hence, we
are searching for an intrinsic characterization of the sequences in
$\mathbb{X}|_\Lambda$. The solution of Problem~\ref{problem2} can help us
solve Problem~\ref{problem1}, because if we have a formula for an equivalent
norm in the trace space, we can characterize the trace space as the space of
sequences having finite norm. This does not address the question of the linear
dependence of the extension. To solve this, it is usual to consider the
next problem.

\begin{problem}
\label{problem3}
Does there exist a linear operator $T: \mathbb{X}|_\Lambda \to \mathbb{X}$
which is bounded and satisfies $(Tf)|_\Lambda = f$, for all $f \in
\mathbb{X}|_\Lambda$? If so, can one give a simple
formula for one such $T$?
\end{problem}

An operator having these properties is usually called a bounded linear extension
operator. In the case when $\Lambda$ is a sequence, the problem of finding an
$F$ such that $F|_\Lambda = f$ can be understood as an interpolation problem
(we say that such an $F$ interpolates the data $f$). A bounded linear extension
operator provides a function $Tf$ interpolating the data $f$ and satisfying
\begin{equation}
\label{eq:quasi-optimal-equivalent}
 \|f\|_{\mathbb{X}|_\Lambda} \leq \|Tf\|_{\mathbb{X}} \leq
\|T\|\|f\|_{\mathbb{X}|_\Lambda}.
\end{equation}
Here the first inequality comes from the definition of the trace norm. Hence, a
bounded linear operator gives an interpolating function $Tf$ having optimal
norm up to a constant factor.
Because of this property, we will refer to
these operators as quasi-optimal interpolators.
In applications, it is important to find interpolators $T$ with small norm.

The construction of a quasi-optimal interpolator given by a simple formula can
also allow us to solve Problem~\ref{problem2}, as
\eqref{eq:quasi-optimal-equivalent} shows that if we put $\|f\|_{\text{eq}} =
\|Tf\|_{\mathbb{X}}$, then $\|\cdot\|_{\text{eq}}$ is an equivalent norm to
$\|\cdot\|_{\mathbb{X}|_\Lambda}$.

In this note we solve Problems~\ref{problem1}--\ref{problem3} for the
Sobolev spaces $W^r_p(I)$, $L^r_p(I)$ for $r = 1,2$; $1 < p < \infty$.
We use the standard notation $f(x_1,\ldots,x_n)$ for the divided differences:
\[
\begin{split}
 f(x_1,x_2) &= \frac{f(x_2) - f(x_1)}{x_2 - x_1},\\
 f(x_1,\ldots,x_n) &= \frac{f(x_2,\ldots,x_n) - f(x_1,\ldots,x_{n-1})}{x_n -
x_1}, \qquad n \geq 3.
\end{split}
\]

We will
always assume that
\begin{equation}
\label{eq:I}
 I = \bigcup_{n\in\mathbb{Z}} [\lambda_n,\lambda_{n+1}].
\end{equation}
This is a general enough case. To see this, let $I = (\alpha,\beta)$ and let
$(a,b)$ be the right hand side of \eqref{eq:I}. Assume, for instance, that
$a = \alpha$, $b <
\beta$ and that we are dealing with the space $L^r_p(I)$ (the other cases are
very similar). Observe that $b < +\infty$, so that
$\Lambda=\{\lambda_n\}_{n\in\mathbb{Z}}$ clusters at $b$.
Assume that $F \in L^r_p[a,b]$ is such that $F|_\Lambda =
f$. Using the fact that
$F^{(j)}$, $0\leq j< r$, are continuous, and the mean value theorem for
divided differences (see \eqref{eq:mean-value}), one sees that the values of
$F(b),\ldots,F^{(r-1)}(b)$ are
completely determined by $f$. Hence, to find a quasi-optimal interpolating
function in $L^r_p(I)$, we first find a quasi-optimal interpolating function
$F \in L^r_p[a,b]$. Then we have to extend this $F$ to a quasi-optimal
$\widetilde{F} \in L^r_p(I)$. Thus, we need to define
$\widetilde{F}|_{(b,\beta)}$,
preserving the continuity of $\widetilde{F}^{(j)}$, $0\leq j < r$. Since the
values
$\widetilde{F}(b),\ldots,\widetilde{F}^{(r-1)}(b)$ are fixed by $f$, one can
even construct
$G = \widetilde{F}|_{(b,\beta)}$ beforehand. Moreover, finding a quasi-optimal
$G \in
L^r_p(b,\beta)$ with given $G(b),\ldots,G^{(r-1)}(b)$ is an easy
problem.

Define
\begin{equation}
\label{eq:equivalent-norms}
\begin{split}
\|f\|^p_{\text{eq},L} &= 
 \sum_{n\in\mathbb{Z}} (\lambda_{n+r} -
\lambda_n)|f(\lambda_n,\ldots,\lambda_{n+r})|^p,\\
\|f\|^p_{\text{eq},W} &= \|f\|^p_{\text{eq},L} +
\sum_{j=0}^{r-1}\sum_{n\in\mathbb{Z}}
(\lambda_{n+r} -
\lambda_n)^{jp+1}|f(\lambda_n,\ldots,\lambda_{n+j})|^p
.
\\
\end{split}
\end{equation}
In this paper we show that
\begin{equation}
\label{eq:equivalent-L}
 C(r)\|f\|_{\text{eq},L} \leq \|f\|_{L_p^r(I)|_\Lambda} \leq C'(r)
\|f\|_{\text{eq},L},\qquad r = 1,2,\quad 1 \leq p < \infty.
\end{equation}
Note that the constants do not depend on $p$.
Let
\[
 h_n = \lambda_{n+1} - \lambda_n, \qquad n \in \mathbb{Z}
\]
be the interpolation steps. We say that the steps are uniformly bounded
if there is a constant $K$ such that
\begin{equation}
\label{eq:bounded-steps}
 h_n \leq K.
\end{equation}
Then we also show that if \eqref{eq:bounded-steps} holds, then
\begin{equation}
\label{eq:equivalent-W}
 C(r,K)\|f\|_{\text{eq},W} \leq \|f\|_{W_p^r(I)|_\Lambda} \leq C'(r,K)
\|f\|_{\text{eq},W},\qquad r = 1,2,\quad 1 \leq p < \infty.
\end{equation}
In fact, we prove a general result (for any $r \geq 1$) which relates
quasi-optimal interpolation in $L^r_p$ and $W^r_p$ (see
Theorem~\ref{theorem-relation}). Using this Theorem, \eqref{eq:equivalent-W}
will follow from \eqref{eq:equivalent-L}.

We conjecture that \eqref{eq:equivalent-L} and \eqref{eq:equivalent-W} are also
true for $r \geq 3$. See Section~\ref{open-questions} for a discussion of this
and some other open questions.

In the case $r=2$, the term $j = 1$ in the expression for $\|f\|_{\text{eq},W}$
can be eliminated, thus giving a simpler expression. We will show in
Proposition~\ref{simpler} that for $r=2$, $\|f\|^p_{\text{eq},W}$ is equivalent
to
\begin{equation}
\label{eq:simple-W}
\|f\|^p_{\text{simp},W} = \|f\|^p_{\text{eq},L} +
\sum_{n\in\mathbb{Z}} (\lambda_{n+1} -
\lambda_{n-1})
|f(\lambda_n)|^p.
\end{equation}
For $r \geq 3$, in general, one cannot hope to remove all the terms $1 \leq j
\leq r-1$ from the expression for $\|f\|_{\text{eq},W}$, thus obtaining
a expression similar to \eqref{eq:simple-W}. However, if the interval $I$ is
large enough in comparison with $K$, this can be done. See
Section~\ref{simplification} for these
results.

Spline interpolators are a useful class of interpolators because they
provide computationally simple formulas. We will say that
$T$ is a spline interpolator of degree $d$ if there is a
decomposition of $I$ into closed intervals $\{I_j\}$ which intersect only at
their endpoints, and such that for any $f$, the
restriction of $Tf$ to each interval $I_j$ is a polynomial of degree at
most $d$. In this paper we construct quasi-optimal spline
interpolators $\Phi_1$, $\Phi_2$ (for $r = 1,2$ respectively) given by simple
expressions.

Extension by smooth functions dates back to Whitney \cite{Whitney}. In 
1934, he solved Problem~\ref{problem1} for the space
$\mathcal{C}^m(\mathbb{R})$.
Brudnyi and Shvartsman extended Whitney's results to the space
$\mathcal{C}^{1,\omega}(\mathbb{R}^n)$ in \cite{BrudnyiShvartsman}.
Fefferman made more progress in \cite{Fefferman05, Fefferman06}, where he
solved Problem~\ref{problem1} for the
spaces $\mathcal{C}^m(\mathbb{R}^n)$ and $\mathcal{C}^{m,1}(\mathbb{R}^n)$. He
also proved the existence of bounded linear extension operators.
Recently, Luli generalized these results to $C^{m,\omega}(\mathbb{R}^n)$ in
\cite{Luli}.

The concept of the depth of an extension operator appears in some of these
works. A linear extension operator $T:\mathbb{X}|_\Lambda \to \mathbb{X}$ is
said to be of bounded depth if there is some $D \geq 0$ such that
\[
 (Tf)(x) = \sum_{y \in \Lambda} \phi(x,y) f(y),
\]
and for each $x$, $\phi(x,y)$ is nonzero for at most $D$ different values of
$y$. If $T$ has bounded depth, its depth is the smallest integer $D$ such that
this holds.

It is easy to see that the spline interpolators constructed in this paper have
bounded depth. In fact, $\Phi_1$ has depth $2$ and $\Phi_2$ has depth $3$.

Fefferman and Klartag have addressed the problem of computing efficiently
the extension function and the norm in the trace space in
\cite{FeffermanKlartagI, FeffermanKlartagII}. See the expository
paper~\cite{Fefferman08}. However, their algorithm does not give
explicit simple formulas.

In \cite{LuliPreprint}, Luli constructs a bounded depth interpolator in
$L_p^r(\mathbb{R})$ by pasting interpolating polynomials using partitions of
unity. Israel gives a result for $L_p^2(\mathbb{R}^2)$ in \cite{Israel}.
Shvartsman has obtained in \cite{Shvartsman07,Shvartsman09} results for the
non-homogeneous Sobolev space $W_p^1$, both in
$\mathbb{R}^n$ and in metric spaces.

In a recent paper \cite{FeffermanIsraelLuliMay}, Fefferman, Israel and Luli
extend Israel's result to $L_p^r(\mathbb{R}^n)$. They show that a
linear extension operator can be constructed such that it has assisted bounded
depth. In general, one cannot hope to construct extension operators of
uniformly bounded depth in $L_p^r$, as they show
in \cite{FeffermanIsraelLuliJune}. In the recent preprint \cite{Shvartsman13},
Shvartsman gives a generalization of the Whitney extension theorem for
$L_p^r(\mathbb{R}^n)$.

We remark that although our setting is less general than those mentioned
above, we will obtain a very simple expression for the norm in the trace
space and explicitly construct a simple quasi-optimal interpolator. We believe
that these results can be of interest in numerical analysis.

In many numerical methods, such as the finite element methods and the Galerkin
methods, an approximation for the
solution of an equation in an infinite dimensional function space is searched
in some finite dimensional subspace. The linear spaces of quasi-optimal splines
we construct (corresponding to a finite number of interpolation nodes) may be
good candidates for these finite dimensional spaces.

\section{The definition of the interpolators $\Phi_1$, $\Phi_2$}
\label{definition-interp}

\subsection{Definition of $\Phi_1$}
The interpolator $\Phi_1$ is just the piecewise linear interpolator. This means
that
$\Phi_1 f$ is an affine function in each interval $[\lambda_n,\lambda_{n+1}]$.
Together with the condition $(\Phi_1 f)(\lambda_n) = f(\lambda_n)$, this
completely determines $\Phi_1$.

The interpolator $\Phi_1$ is given by the following formula:
\begin{equation}
 \label{eq:formula-phi1}
(\Phi_1 f)(x) = f(\lambda_n)\frac{\lambda_{n+1}-x}{h_n} +
f(\lambda_{n+1})\frac{x-\lambda_n}{h_n}, \qquad \lambda_n \leq x \leq
\lambda_{n+1}.
\end{equation}

\subsection{Definition of $\Phi_2$}
The interpolator $\Phi_2$ is a spline interpolator of degree 3. It is defined
as follows.

Let 
\[
\mu_n = \frac{\lambda_n + \lambda_{n+1}}{2}.
\]
Given $f: \Lambda \to \mathbb{C}$, we
construct cubic polynomials on each of the intervals $[\mu_{n-1},
\lambda_n]$ and $[\lambda_n, \mu_n]$ such that the resulting cubic spline
$\Phi_2 f$ is $\mathcal{C}^1$ and satisfies $(\Phi_2 f)(\lambda_n) =
f(\lambda_n)$.

Put
\[
 \alpha_n(f) = \frac{h_nf(\lambda_{n-1},\lambda_n) +
h_{n-1}f(\lambda_n,\lambda_{n+1})}{h_{n-1}+h_n}.
\]
The conditions
\begin{equation}
\label{eq:phi2-conditions}
\begin{split}
 (\Phi_2 f)(\lambda_n) &= f(\lambda_n),\qquad (\Phi_2 f)'(\lambda_n)
= \alpha_n(f)\\
 (\Phi_2 f)(\mu_n) &= \frac{f(\lambda_n) + f(\lambda_{n+1})}{2},\qquad (\Phi_2
f)'(\mu_n) = f(\lambda_n,\lambda_{n+1}).
\end{split}
\end{equation}
determine $\Phi_2 f$, since $\Phi_2 f$ is a piecewise cubic polynomial
on each of the
intervals $[\mu_{n-1},\lambda_n]$, $[\lambda_n,\mu_n]$. Observe that, restricted
to the interval $[\mu_{n-1},\mu_n]$, $\Phi_2 f$ only depends on
$f(\lambda_{n-1})$, $f(\lambda_n)$, $f(\lambda_{n+1})$.

\begin{figure}
\begin{center}
\includegraphics[width=12cm]{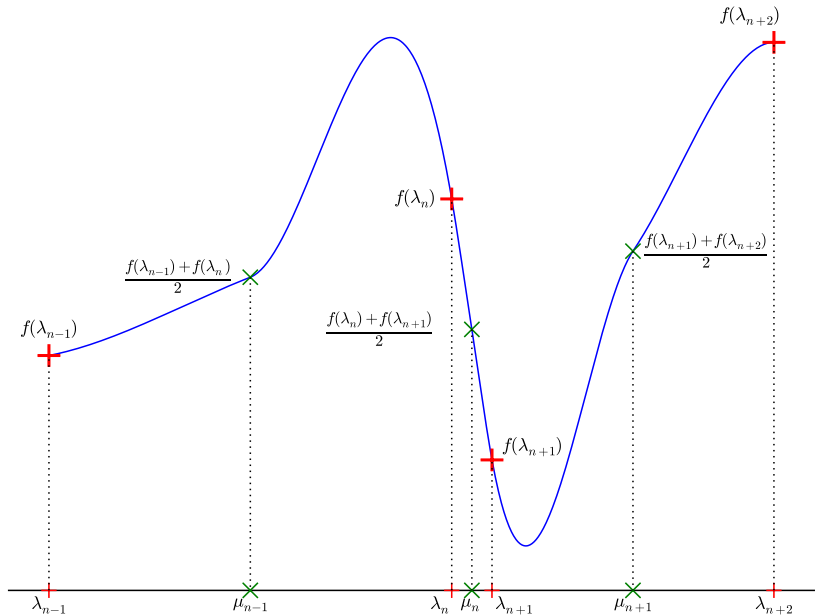}
\caption{Graph of $\Phi_2 f$.}
\label{fig:phi2}
\end{center}
\end{figure}

See Figure~\ref{fig:phi2} for a typical graph of $\Phi_2 f$. The
interpolation nodes $(\lambda_n, f(\lambda_n))$ are shown as $+$'s and the
auxiliary nodes $(\mu_n, (f(\lambda_n)+f(\lambda_{n+1}))/2)$ are shown as
$\times$'s.

We also have the following explicit formula for $\Phi_2 f$. Define
\[
 q(x) = 4(x^2 - x^3)
\]
and observe that $q$ satisfies the boundary conditions
\begin{equation}
\label{eq:q-conditions}
 q(0) = q'(0) = 0,\qquad q\left(\tfrac{1}{2}\right) = \tfrac{1}{2},\qquad
 q'\left(\tfrac{1}{2}\right) = 1.
\end{equation}
Then we have
\begin{equation}
\label{eq:phi2-formula}
\begin{split}
(\Phi_2 f)(x) &= f(\lambda_n) + \alpha_n(f)(x-\lambda_n)
+
h_{n-1}^2f(\lambda_{n-1},\lambda_n,\lambda_{n+1})
q\left(\frac {\lambda_n - x}{h_{n-1}}\right),
\quad \mu_{n-1} \leq x \leq
\lambda_n,\\
(\Phi_2 f)(x) &= f(\lambda_n) + \alpha_n(f)(x-\lambda_n)
+
h_n^2f(\lambda_{n-1},\lambda_n,\lambda_{n+1})q\left(\frac{x -
\lambda_n}{h_n}\right),\quad \lambda_n \leq x \leq \mu_n.
\end{split}
\end{equation}
By using the identities
\[
\begin{split}
h_{n-1}^2f(\lambda_{n-1},\lambda_n,\lambda_{n+1}) &= f(\lambda_{n-1}) -
f(\lambda_n) + h_{n-1}\alpha_n(f),\\
h_n^2f(\lambda_{n-1},\lambda_n,\lambda_{n+1}) &=
f(\lambda_{n+1}) - f(\lambda_n)- h_n\alpha_n(f),
\end{split}
\]
and \eqref{eq:q-conditions}, it is easy to see that
\eqref{eq:phi2-formula} defines the same spline as
\eqref{eq:phi2-conditions}.

\section{Main results}
\label{main-results}

In the sequel, we will use the notation $\|f\|_1 \approx \|f\|_2$,
which means that the norms $\|\cdot\|_1$ and $\|\cdot\|_2$ are equivalent, i.e.,
there are constants $C > 0$ and $C' < \infty$ such that, for all $f$,
\[
 C\|f\|_1 \leq \|f\|_2 \leq C'\|f\|_1.
\]

The first
result solves Problems~\ref{problem1}--\ref{problem3} for $L^r_p(I)$,
$W^r_p(I)$, $r = 1,2$.

\begin{theorem}
\label{main}
Let $r = 1,2$, and $1 \leq p < \infty$.
Define $\|f\|^p_{\text{eq},W}$, $\|f\|^p_{\text{eq},L}$ from
\eqref{eq:equivalent-norms}.
Then the following statements are true:
\begin{enumerate}[(a)]
 \item
 \label{t1-a}
 The seminorm $\|\cdot\|_{\text{eq},L}$ gives an equivalent seminorm in
$L_p^r(I)$:
\[
 \|f\|_{L_p^r(I)|_\Lambda} \approx \|f\|_{\text{eq},L},
\]
where the equivalence constants depend only on $r$. Moreover,
$\Phi_r:L_p^r(I)|_\Lambda \to L_p^r(I)$ is quasi-optimal and
$\|\Phi_r\| \leq C(r)$.

\item
\label{t1-b}
If \eqref{eq:bounded-steps} holds,
then the norm $\|\cdot\|_{\text{eq},W}$ gives an equivalent norm in
$W_p^r(I)$:
\[
 \|f\|_{W_p^r(I)|_\Lambda} \approx \|f\|_{\text{eq},W}.
\]
The equivalence constants depend only on $r,K$.
Moreover, $\Phi_r:W_p^r(I)|_\Lambda \to W_p^r(I)$ is quasi-optimal and
$\|\Phi_r\| \leq C'(r,K)$.
\end{enumerate}
\end{theorem}

If the assumption \eqref{eq:bounded-steps} in this Theorem is not fulfilled,
we can always use the following trick. First
fix
some $L > 0$. Let $\{J_k\}$ be the collection of all the intervals
$J_k = [\lambda_k,
\lambda_{k+1}]$ such that $h_k > L$. Given a function $f$ on $\Lambda$,
proceeding as if \eqref{eq:bounded-steps} were true, one obtains an
interpolating
function
$F$ on $I$.
However, $\|F\|_{W^r_p(I)}$ is not comparable with
$\|f\|_{W^r_p(I)|_{\Lambda}}$.

Now one chooses $\mathcal{C}^\infty(\mathbb{R})$ functions $\varphi_k$ such
that $\|\varphi_k\|_{\mathcal{C}^r(\mathbb{R})} \leq M$, $0 \leq \varphi_k(x)
\leq 1$, $\varphi_k(x) = 1$ if $x \notin
\mathring{J_k}$ and $|\operatorname{supp}({\varphi_k}_{|J_k})| \leq L$. Put
$\varphi = \prod \varphi_k$. Then one can see that
$\|\varphi F\|_{W^r_p(I)}$ will be comparable with
$\|f\|_{W^r_p(I)|_{\Lambda}}$.

The second result relates quasi-optimal interpolation on $L^r_p$ and $W^r_p$.
It will be used to deduce part (\ref{t1-b}) of the Theorem above from part
(\ref{t1-a}). It could also be of help in attacking the problem for $r \geq 3$.

\begin{theorem}
\label{theorem-relation}
Let $r \geq 1$ be arbitrary and $1 \leq p < \infty$. Suppose that
\eqref{eq:bounded-steps} holds. The
following statements are true:
\begin{enumerate}[(a)]
 \item
 \label{trr-a}
 Assume that the following equivalence of norms holds:
 \begin{equation}
 \label{eq:trr-*}
  \|f\|_{\text{eq},L} \approx \|f\|_{L^r_p(I)|_\Lambda},\qquad f\in
L^r_p(I)|_\Lambda.
 \end{equation}
 Then the following equivalence of norms also holds:
 \[
 \|f\|_{\text{eq},W} \approx \|f\|_{W^r_p(I)|_\Lambda},\qquad
f\in W^r_p(I)|_\Lambda.
 \]
 The equivalence constants depend only on $r,K$ and the equivalence constants
 in \eqref{eq:trr-*}.
 
 \item
 \label{trr-b}
 If $T: L^r_p(I)|_\Lambda \to L^r_p(I)$ is a quasi-optimal interpolator, then
$T: W^r_p(I)|_\Lambda \to W^r_p(I)$ is also a quasi-optimal interpolator.
\end{enumerate}
\end{theorem}

The proofs of Theorems \ref{main} and \ref{theorem-relation} will be given at
the end of the next section.

\section{Proof of the results}

We start by proving that
\begin{equation}
\label{eq:one-sided}
 \|f\|_{\text{eq},X} \leq C\|f\|_{X^r_p(I)|_\Lambda},\qquad X = L,W.
\end{equation}
This will follow from a few easy estimates and is valid for all $r \geq 1$.

\begin{lemma}
\label{f-lower-bound}
Let $r\geq 1$, and $J \subset \mathbb{R}$ be a closed interval with $|J| \leq
K$. Let $x_1,\ldots,x_{k+1} \in J$, $0\leq k \leq r-1$. Then, for
any $F \in W_p^r(J)$,
\[
 \|F\|_{W_p^r(J)}^p \geq C^p|J|^{kp+1}|F(x_1,\ldots,x_{k+1})|^p,
\]
where $C > 0$ is a constant $C = C(r,K)$.
\end{lemma}

\begin{proof}
By the mean value theorem for divided differences (see, for instance,
\cite[Chapter 3.2]{Atkinson}), there is some $\xi \in J$
such that
\begin{equation}
\label{eq:mean-value}
 F^{(k)}(\xi) = \frac{F(x_1,\ldots,x_{k+1})}{k!}.
\end{equation}

It is an easy consequence of the
Sobolev embedding theorem that there is some universal constant $C_0 > 0$
such that if $G \in W_p^r[0,1]$ and $\eta \in [0,1]$, then
\[
 \|G\|_{W_p^r[0,1]} \geq C_0 |G^{(k)}(\eta)|.
\]

Let $\varphi(x)$ be the affine transformation
taking $[0,1]$ to $J$. If $F \in W_p^r(J)$, put $G(x) = F(\varphi(x))$
and compute
\[
\begin{split}
\frac{1}{|J|} \|F\|_{W_p^r(J)}^p &=
\frac{1}{|J|}\int_J \left(
|F(x)|^p + |F^{(r)}(x)|^p \right)dx =
\int_0^1 \left(
|G(x)|^p + |J|^{-rp}|G^{(r)}(x)|^p \right)dx \\
&\geq
C_1^p\|G\|_{W_p^r[0,1]}^p \geq C_1^pC_0^p|G^{(k)}(\varphi^{-1}(\xi))|^p =
C_1^pC_0^p|J|^{kp}|F^{(k)}(\xi)|^p \\ &\geq
C_1^pC_0^p[(r-1)!]^{-p}|J|^{kp}|F(x_0,\ldots,x_{k+1})|^p
\end{split}
\]
where $C_1 = \min\{1,K^{-r}\}$. The lemma holds with
$C = C_1C_0/(r-1)!$.
\end{proof}

Let us now recall the following result from \cite{LuliPreprint}.

\begin{lemmacite}[see \cite{LuliPreprint}]
\label{df-lower-bound}
Let $r \geq 1$, $x_1 < x_2 < \cdots < x_{r+1}$, $J = [x_1,x_{r+1}]$, $F \in
L_p^r(J)$. Then
\[
 |F(x_1,x_2\ldots,x_{r+1})|^p|J| \leq [(r-1)!]^{-p}\|F\|_{L_p^r(J)}^p.
\]
\end{lemmacite}

With these two Lemmas, now we can prove the following Proposition, which
will give \eqref{eq:one-sided}.

\begin{proposition}
\label{lower}
Let $r \geq 1$. 
Put $f = F|_\Lambda$ and define $\|f\|_{\text{eq},L}$ and
$\|f\|_{\text{eq},W}$ from \eqref{eq:equivalent-norms}.
The following statements hold:
\begin{enumerate}[(a)]
\item \label{norm-lower-homog}
If $F \in L_p^r(I)$, then
$
 \|F\|_{L_p^r(I)} \geq C\|f\|_{\text{eq},L}
$
for some constant $C > 0$ depending only on $r$.
 \item \label{norm-lower}
If \eqref{eq:bounded-steps} holds and $F \in W_p^r(I)$, then
$
 \|F\|_{W_p^r(I)} \geq C'\|f\|_{\text{eq},W}
$
for some constant $C' > 0$ depending only on $r,K$.
\end{enumerate}
\end{proposition}

\begin{proof}
If $F \in L_p^r(I)$, we have
\[
\|F\|^p_{L_p^r(I)} =
\frac{1}{r}\sum_{n\in\mathbb{Z}}\|F\|^p_{L_p^r[\lambda_n,\lambda_{n+r}]} \geq
\frac{[(r-1)!]^p}{r}\sum_{n\in\mathbb{Z}}(\lambda_{n+r} -
\lambda_n)|F(\lambda_n,\ldots, \lambda_ { n+r }
)|^p,
\]
by Lemma~\ref{df-lower-bound}. This implies (\ref{norm-lower-homog}), with
$C = (r-1)!/r$.

To prove (\ref{norm-lower}), fix $k = 0,1,\ldots,r-1$ and use
Lemma~\ref{f-lower-bound} to obtain
\[
 \|F\|^p_{W_p^r(I)} = \frac{1}{r}
\sum_{n\in\mathbb{Z}}\|F\|^p_{W_p^r[\lambda_n,\lambda_{n+r}]} \geq
\frac{C''^p}{r}\sum_{n\in\mathbb{Z}} (\lambda_{n+r} -
\lambda_n)^{kp+1}|F(\lambda_n,\ldots,\lambda_{n+k})|^p.
\]
By summing these inequalities over $k = 0,\ldots,r-1$ together with
(\ref{norm-lower-homog}), we obtain (\ref{norm-lower}).
\end{proof}

Our goal now is to prove Theorem~\ref{theorem-relation}. We will need the
following Lemma, which is related to Friedrichs'
inequality.

\begin{lemma}
\label{easy-lemma}
Let $J \subset \mathbb{R}$ be an interval of finite length. If $F
\in L^r_p(J)$ and $\xi_0,\ldots,\xi_{r-1} \in J$, then
\[
\|F\|^p_{L^p(J)} \leq C^p \left[
|J|^{rp}\|F\|^p_{L^r_p(J)} +
\sum_{j=0}^{r-1} |J|^{jp+1}|F^{(j)}(\xi_j)|^p \right]
,
\]
for some constant $C = C(r)$.
\end{lemma}
\begin{proof}
We have, for $j = 0,\ldots,r-1$,
\[
 F^{(j)}(x) = F^{(j)}(\xi_j) + \int_{\xi_j}^x F^{(j+1)}(t)\,dt,
\]
so that
\[
|F^{(j)}(x)| \leq |F^{(j)}(\xi_j)| + \int_{J} |F^{(j+1)}(t)|\,dt.
\]

Repeated application of this inequality yields
\[
\begin{split}
 |F(x)| &\leq |F(\xi_0)| + \int_{J} |F'(t)|\,dt 
 \leq |F(\xi_0)| + |J||F'(\xi_1)| + |J|\int_{J} |F''(t)|\,dt \\
 &\leq \cdots
\leq
 \sum_{j=0}^{r-1} |J|^j|F^{(j)}(\xi_j)| + |J|^{r-1}\int_{J}
|F^{(r)}(t)|\,dt
\leq
 \sum_{j=0}^{r-1} |J|^j|F^{(j)}(\xi_j)| + |J|^{r-1
+\frac{p-1}{p}}\|F\|_{L^r_p(J)},
\end{split}
\]
where the last inequality comes from H\"older's inequality.
Using H\"older's inequality for sums,
we see that
\[
 |F(x)|^p \leq (r+1)^{p-1} \left[
 \sum_{j=0}^{r-1} |J|^{jp}|F^{(j)}(\xi_j)|^p + |J|^{rp - 1}\|F\|^p_{L^r_p(J)}
 \right].
\]
Now the lemma follows by integrating this inequality.
\end{proof}

Now we can give the proof of Theorem~\ref{theorem-relation}.

\begin{proof}[Proof of Theorem~\ref{theorem-relation}]

Let $F \in W^r_p(I)$ and 
fix $n \in \mathbb{Z}$. By the mean value theorem for divided differences, we
can choose $\xi_0,\ldots,\xi_{r-1} \in [\lambda_n,\lambda_{n+r}]$ such that
\[
 F^{(j)}(\xi_j) = \frac{F(\lambda_n,\ldots,\lambda_{n+j})}{j!},\qquad
j=0,\ldots,r-1.
\]
By Lemma~\ref{easy-lemma},
\[
\begin{split}
 \|F\|^p_{L^p[\lambda_n,\lambda_{n+r}]} &\leq C_0^p \left[
 (\lambda_{n+r} - \lambda_n)^{rp}\|F\|^p_{L^r_p[\lambda_n,\lambda_{n+r}]}
 + \sum_{j=0}^{r-1}
(\lambda_{n+r}-\lambda_n)^{jp+1}\frac{|F(\lambda_n,\ldots,\lambda_{n+j})|^p}{j!}
\right]
\\
&\leq
C_1^p \left[
\|F\|^p_{L^r_p[\lambda_n,\lambda_{n+r}]}
 + \sum_{j=0}^{r-1}
(\lambda_{n+r}-\lambda_n)^{jp+1}|F(\lambda_n,\ldots,\lambda_{n+j})|^p
\right],
\end{split}\]
where $C_0 = C_0(r)$ comes from Lemma~\ref{easy-lemma} and
$C_1 = C_1(r,K) = C_0\cdot\max\{1,(rK)^{r}\}$
(see \eqref{eq:bounded-steps}). By summing over $n
\in \mathbb{Z}$, using $F|_\Lambda = f$, and by the definition of
$\|f\|_{\text{eq},W}$, we get
\begin{equation}
\label{eq:theorem-relation-*}
\|F\|^p_{L^p(I)} \leq C_1^p \left[
\|F\|^p_{L^r_p(I)} + \|f\|^p_{\text{eq},W}\right].
\end{equation}

Now assume that \eqref{eq:trr-*} holds. Given $f \in
W^r_p(I)|_\Lambda$, choose $F \in L^r_p(I)$ with $F|_\Lambda = f$ and
$\|F\|_{L^r_p(I)} \leq 2\|f\|_{L^r_p(I)|_\Lambda}$.
Then,
\[
 \|F\|^p_{L^r_p(I)} \leq 2^p\|f\|^p_{L^r_p(I)|_\Lambda} \leq
C^p\|f\|^p_{\text{eq},L} \leq C^p\|f\|^p_{\text{eq},W},
\]
by Proposition~\ref{lower}. Applying \eqref{eq:theorem-relation-*}, we have
\[
\begin{split}
 \|F\|^p_{W^r_p(I)} &=
\|F\|^p_{L^p(I)} + \|F\|^p_{L^r_p(I)} \leq (C_1^p+1)\|F\|^p_{L^r_p(I)} +
C_1^p\|f\|^p_{\text{eq},W}
 \leq (C^pC_1^p + C^p +
C_1^p)\|f\|^p_{\text{eq},W}.
\end{split}
\]

Hence, by definition of the trace norm, one obtains
\[
 \|f\|^p_{W^r_p(W)|_\Lambda} \leq (C^pC_1^p + C^p +
C_1^p)\|f\|^p_{\text{eq},W}.
\]
The reverse inequality comes from
Proposition~\ref{lower}, so that this proves (\ref{trr-a}).

To prove (\ref{trr-b}), assume that $T:L^r_p(I)|_\Lambda \to L^r_p(I)$ is
quasi-optimal and let $\|T\|$ be the norm of $T$ as an operator
$L^r_p(I)|_\Lambda \to L^r_p(I)$. Observe that $W^r_p(I)|_\Lambda \subset
L^r_p(I)|_\Lambda$. Also, if $f \in L^r_p(I)|_\Lambda$, $Tf$ is locally in
$W^r_p$, because $Tf \in L^r_p(I)$. Hence, $T$ will make sense as an operator
taking $W^r_p(I)|_\Lambda$ into $W^r_p(I)$ if we can prove that
$\|Tf\|_{W^r_p(I)}$ is finite whenever $f \in W^r_p(I)|_\Lambda$.

If $f \in
W^r_p(I)|_\Lambda$, we use
\eqref{eq:theorem-relation-*} and Proposition~\ref{lower} to obtain
\[
\begin{split}
 \|Tf\|^p_{W^r_p(I)} &=
\|Tf\|^p_{L^p(I)} + \|Tf\|^p_{L^r_p(I)} \leq (C_1^p+1)\|Tf\|^p_{L^r_p(I)} +
C_1^p\|f\|^p_{\text{eq},W}\\
&\leq
(C_1^p+1)\|T\|^p\|f\|^p_{L^r_p(I)|_\Lambda} +
C_1^pC'^{-p}\|f\|^p_{W^r_p(I)|_\Lambda} \\ &\leq 
((C_1^p+1)\|T\|^p + C_1^pC'^{-p})\|f\|^p_{W^r_p(I)|_\Lambda}.
\end{split}
\]
Here we have used $\|f\|_{L^r_p(I)|_\Lambda} \leq
\|f\|_{W^r_p(I)|_\Lambda}$, which is trivial. This shows that
$T:W^r_p(I)|_\Lambda \to W^r_p(I)$ is quasi-optimal.
\end{proof}

Now we can give the proof of Theorem~\ref{main}.

\begin{proof}[Proof of Theorem~\ref{main}]
Observe that for any $f : \Lambda \to \mathbb{C}$, the function $\Phi_r f$ is
locally in $L^r_p$, because it is of class $\mathcal{C}^{r-1}$ and piecewise
$\mathcal{C}^r$. Hence, to check that $\Phi_r(L^r_p(I)|_\Lambda) \subset
L^r_p(I)$, it is enough to see that $\|\Phi_r f\|_{L^r_p(I)}$ is finite for any
$f \in L^r_p(I)|_\Lambda$.

Therefore, we only need to prove that
\begin{equation}
\label{eq:needed-ineq}
 \|\Phi_r f\|_{L^r_p(I)} \leq C(r)\|f\|_{\text{eq},L},\qquad r = 1,2.
\end{equation}
Then we will have
\[
 \|f\|_{L^r_p(I)|_\Lambda} \leq \|\Phi_r f\|_{L^r_p(I)} \leq
C(r)\|f\|_{\text{eq},L} \leq C'(r)\|f\|_{L^r_p(I)|_\Lambda},
\]
where the last inequality comes from
Proposition~\ref{lower}, so (\ref{t1-a}) will follow. Using
Theorem~\ref{theorem-relation}, (\ref{t1-b}) follows from (\ref{t1-a}), because
the fact that
\[\|\Phi_r\|_{W^r_p(I)|_\Lambda \to W^r_p(I)} \leq C'(r,K),\qquad r = 1,2\]
is contained in its proof.

Inequality \eqref{eq:needed-ineq} will follow from some easy
computations
using the formulas given in Section~\ref{definition-interp}. The analogous
estimate for
$W^r_p(I)$ can also be obtained directly by the same means instead of appealing
to Theorem~\ref{theorem-relation}.

We first deal with the case $r = 1$.
We compute $\|\Phi_1 f\|_{L^1_p[\lambda_n,\lambda_{n+1}]}$ using formula
\eqref{eq:formula-phi1}:
\[
 \|\Phi_1 f\|_{L^1_p[\lambda_n,\lambda_{n+1}]}^p =
|f(\lambda_n,\lambda_{n+1})|^p\|1\|^p_{L^p[\lambda_n,\lambda_{n+1}]} =
 (\lambda_{n+1}-\lambda_n)|f(\lambda_n,\lambda_{n+1})|^p.
\]
By summing over $n \in \mathbb{Z}$, we obtain
\[
 \|\Phi_1 f\|_{L^1_p(I)}^p = \|f\|^p_{\text{eq},L}.
\]

For the case $r = 2$, we compute
$\|\Phi_2 f\|^p_{L^2_p[\lambda_n,\mu_n]}$ using
formula \eqref{eq:phi2-formula}. We have, for $x \in [\lambda_n,\mu_n]$,
\[
 (\Phi_2 f)''(x) =
f(\lambda_{n-1},\lambda_n,\lambda_{n+1})q''\left(\frac{x-\lambda_n}{h_n}\right).
\]
Making the change of variables $y = (x-\lambda_n)/h_n$, we compute
\[
\left\|q''\left(\tfrac{x-\lambda_n}{h_n}\right)\right\|^p_{L^p[\lambda_n,\mu_n]}
=
h_n\|q''\|^p_{L^p[0,1/2]}
\leq
2^{p-1}\|q''\|^p_{L^1[0,1/2]} h_n
\leq M^p h_n,
\]
where we have used Jensen's inequality, and 
$M = 2\|q''\|_{L^1[0,1/2]}$ is some universal constant.

Hence,
\begin{equation}
\label{eq:2-+}
\|\Phi_2 f\|^p_{L^2_p[\lambda_n,\mu_n]} =
M^ph_n|f(\lambda_{n-1},\lambda_n,\lambda_{n+1})|^p.
\end{equation}
Proceeding analogously, we obtain
\begin{equation}
\label{eq:2-++}
\|\Phi_2 f\|^p_{L^2_p[\mu_{n-1},\lambda_n]} =
M^ph_{n-1}|f(\lambda_{n-1},\lambda_n,\lambda_{n+1})|^p.
\end{equation}
By summing inequalities \eqref{eq:2-+} and \eqref{eq:2-++} over $n \in
\mathbb{Z}$ and using $h_{n-1} + h_n = \lambda_{n+1} - \lambda_{n-1}$, we find
that
\[
 \|\Phi_2 f\|^p_{L^2_p(I)} = M^p\|f\|^p_{\text{eq}, L}.
\]
This proves the Theorem.
\end{proof}

\section{Simplification of the expression for $\|f\|_{\mathrm{eq},W}$}
\label{simplification}

In this Section, we examine whether the expression for $\|f\|_{\text{eq},W}$
given in \eqref{eq:equivalent-norms} can be simplified by eliminating the terms
corresponding to $1 \leq j \leq r-1$. Throughout the Section we will assume
that \eqref{eq:bounded-steps} holds.

The next proposition proves that the answer is affirmative for $r =
2$. Its proof reduces to doing some easy but somewhat lengthy estimates on the
divided differences.

\begin{proposition}
 \label{simpler}
Let $r = 2$ and assume that \eqref{eq:bounded-steps} holds. Define
$\|f\|_{\text{eq},W}$ and $\|f\|_{\text{simp},W}$ from
\eqref{eq:equivalent-norms}
and \eqref{eq:simple-W}. Then we have the equivalence of norms
\[
 \|f\|_{\text{eq},W} \approx \|f\|_{\text{simp},W}.
\]
The equivalence constants depend only on $K$.
\end{proposition}
\begin{proof}
First, we have
\[
 |f(\lambda_n)| \leq |f(\lambda_{n-1})| + (\lambda_{n+1} -
\lambda_{n-1})|f(\lambda_{n-1},\lambda_n)|.
\]
Hence,
\[
(\lambda_{n+1}-\lambda_{n-1})|f(\lambda_n)|^p \leq 2^{p-1} [
(\lambda_{n+1}-\lambda_{n-1})|f(\lambda_{n-1})|^p + (\lambda_{n+1} -
\lambda_{n-1})^{p+1}|f(\lambda_{n-1},\lambda_n)|],
\]
by H\"older's inequality for sums.
Adding over $n \in \mathbb{Z}$, we obtain
$
\|f\|^p_{\text{simp},W} \leq 2^{p-1} \|f\|^p_{\text{eq},W}.
$

To prove the reverse inequality, we fix $n \in \mathbb{Z}$.
Since $\lambda_{n+2} - \lambda_n = (\lambda_{n+2} - \lambda_{n+1}) +
(\lambda_{n+1} - \lambda_n)$, we have either $\lambda_{n+2} - \lambda_{n+1} \geq
(\lambda_{n+2} - \lambda_n)/2$ or $\lambda_{n+1} - \lambda_n \geq
(\lambda_{n+2} - \lambda_n)/2$. Assume that $\lambda_{n+2} - \lambda_{n+1} \geq
(\lambda_{n+2} - \lambda_n)/2$. The other case is similar and easier.

By the triangle inequality and the inequality $|x+y|^p \leq 2^{p-1}[|x|^p +
|y|^p]$, one obtains the following inequalities:
\begin{gather}
\label{eq:I2}
|f(\lambda_n,\lambda_{n+1})|^p \leq 2^{p-1}[
(\lambda_{n+2}-\lambda_n)^p|f(\lambda_n,\lambda_{n+1},\lambda_{n+2})|^p
+
|f(\lambda_{n+1},\lambda_{n+2})|^p],
\\
\label{eq:I3}
(\lambda_{n+2}-\lambda_n)^p|f(\lambda_{n+1},\lambda_{n+2})|^p \leq 2^{p-1}[
|f(\lambda_{n+1})|^p + |f(\lambda_{n+2})|^p],\\
\label{eq:I4}
(\lambda_{n+2}-\lambda_n)|f(\lambda_{n+2})|^p \leq
2(\lambda_{n+2}-\lambda_{n+1})|f(\lambda_{n+2})|^p.
\end{gather}

By applying inequalities \eqref{eq:I2}, \eqref{eq:I3}, \eqref{eq:I4}, one gets
\begin{equation}
\label{eq:foo}
\begin{split}
 (\lambda_{n+2}-&\lambda_n)^{p+1}|f(\lambda_n,\lambda_{n+1})|^p \leq
 2^{2p-2}(\lambda_{n+2}-\lambda_n)|f(\lambda_{n+1})|^p \\ &+
2^{2p-1}(\lambda_{n+3}-\lambda_{n+1})|f(\lambda_{n+2})|^p
+
2^{p-1}(\lambda_{n+2}-\lambda_n)^{2p+1}|f(\lambda_n,\lambda_{n+1},\lambda_{n+2}
)|^p.
\end{split}
\end{equation}
Hence, we have
\begin{equation}
\label{eq:simeq-*}
\begin{split}
 (\lambda_{n+2}&-\lambda_n)|f(\lambda_n)|^p + 
 (\lambda_{n+2}-\lambda_n)^{p+1}|f(\lambda_n,\lambda_{n+1})|^p \leq \\ &\leq
 2^{p-1}(\lambda_{n+2}-\lambda_n)|f(\lambda_{n+1})|^p +
 (2^{p-1}+1)(\lambda_{n+2}-\lambda_n)^{p+1}|f(\lambda_n,\lambda_{n+1})|^p\\
 &\leq
 C(K)^p\Big[
 (\lambda_{n+2}-\lambda_n)|f(\lambda_{n+1})|^p +
 (\lambda_{n+3}-\lambda_{n+1})|f(\lambda_{n+2})|^p \\
 &\qquad{}\qquad{}+
 (\lambda_{n+2}-\lambda_n)|f(\lambda_n,\lambda_{n+1},\lambda_{n+2})|^p
\Big],
\end{split}
\end{equation}
where in the last inequality we have used
\eqref{eq:foo} and \eqref{eq:bounded-steps}.

In the case $\lambda_{n+1} - \lambda_n \geq
(\lambda_{n+2} - \lambda_n)/2$, one obtains
\begin{equation}
\label{eq:simeq-**}
\begin{split}
 (\lambda_{n+2}-\lambda_n)|f(\lambda_n)|^p &+
(\lambda_{n+2}-\lambda_n)^{p+1}|f(\lambda_n,\lambda_{n+1})|^p \leq\\
&\leq C'^p[(\lambda_{n+1}-\lambda_{n-1})|f(\lambda_n)|^p +
(\lambda_{n+2}-\lambda_n)|f(\lambda_{n+1})|^p].
\end{split}
\end{equation}
In fact, \eqref{eq:bounded-steps} is not necessary in this case.

For every index $n\in\mathbb{Z}$, we have either \eqref{eq:simeq-*} or
\eqref{eq:simeq-**}. By summing these inequalities over $n\in\mathbb{Z}$, we
obtain $\|f\|^p_{\text{eq},W} \leq C''(K)^p \|f\|^p_{\text{simp},W}$.
\end{proof}

Now we give an example that shows that, in general, for any $r \geq 3$, one
cannot eliminate all the terms $1 \leq j \leq r-1$ from the expression of
$\|f\|_{\text{eq},W}$. Assume that $|I| < \infty$ and suppose that one
wishes to prove the following equivalence of norms:
\begin{equation}
\label{eq:impossible-equivalence}
\|f\|^p_{W^r_p(I)|_\Lambda} \approx \|f\|^p_{\text{eq},L} +
\sum_{n\in\mathbb{Z}}
c_n|f(\lambda_n)|^{p},
\end{equation}
for some coefficients $\{c_n\}$ not depending on $f$, and with equivalence
constants depending only on $r,p,K$. Taking $f \equiv 1$ we see that
the coefficients $\{c_n\}$ must be in $\ell^1$ and satisfy
$\|\{c_n\}\|_1 \leq C(r,p,K)|I|$. Then, the next example shows that
\eqref{eq:impossible-equivalence} cannot hold in some cases.

\begin{example}
Let $r \geq 3$. Fix a small $h > 0$ and let $\Lambda =
\{\lambda_n\}_{n\in\mathbb{Z}}$
be a strictly increasing sequence such that $\lambda_0 = -1$, $\lambda_1 = 1$,
$\lambda_n \to 1+h$ as $n \to +\infty$ and $\lambda_n \to -1-h$ as $n \to
-\infty$. This sequence $\Lambda$ satisfies \eqref{eq:bounded-steps} with $K =
2$.

Put
$
 p(x) = [(1+h)^2 - x^2]/[h(2+h)]
$
and consider the interpolation problem given by $f(\lambda_n) = p(\lambda_n)$.
Then $F(x) = p(x)$ solves this problem.
Now observe that $\|p\|_{L^r_p(I)} = 0$, $\|p\|_{W^r_p(I)} \approx 1/h$ and
$|p(\lambda_n)| \leq 1$ for all $n \in \mathbb{Z}$. Hence, the right hand side
of \eqref{eq:impossible-equivalence} is less or equal than $C(r,p,K)|I|$.

It follows from \eqref{eq:theorem-relation-*} and Proposition~\ref{lower} that
\[
\|p\|^p_{W^r_p(I)} \leq C'(r,p,K)\|f\|^p_{\text{eq},W} \leq
C''(r,p,K)\|f\|^p_{W^r_p(I)|_\Lambda}.
\]
Hence, the left hand side of \eqref{eq:impossible-equivalence} is greater or
equal than $C''(r,p,K)^{-1}/h^p$. This is a contradiction for small enough $h >
0$.
\end{example}

However, the former example is in some sense pathological because it does not
satisfy
the property that one can choose $r+1$ points $\lambda_n$ which are ``well
separated''. It turns out that this kind of condition is the only thing one
needs to be able to eliminate the terms $1 \leq j \leq r-1$ in the expression
for $\|f\|_{\text{eq},W}$. This always happens if the interval $I$ is large
enough. In fact, one can prove the following Proposition.

\begin{proposition}
Let $r \geq 1$, $1 \leq p < \infty$ and assume that $|I| \geq (4r+2)K$
\emph{(}see \eqref{eq:bounded-steps}\emph{)}. Define $\|f\|_{\text{simp},W}$
from \eqref{eq:simple-W}. Then Theorem~\ref{theorem-relation} remains true if
one replaces $\|f\|_{\text{eq},W}$ with $\|f\|_{\text{simp},W}$.
\end{proposition}

We will only give a sketch of the proof.

First one can prove that if $J$ is an interval with
$|J| = (4r+2)K$, and $\eta_k\in J$, $k=1,\ldots,r+1$
satisfy
$|\eta_k - \eta_l|\geq K$ for $k \neq l$, then
\[
 \|F\|^p_{L^p(J)} \leq C(r,K)^p\left[ \|F\|^p_{L^r_p(J)} + \sum_{k=1}^{r+1}
|F(\eta_k)|^p\right].
\]

Now one covers $I$ with intervals $I_j$ of length $(4r+2)K$ in such a way that
the interiors of any three of these intervals do not intersect. One obtains the
above inequality for $J = I_j$.
By averaging the above estimate among all possible choices of $\eta_k =
\lambda_{n_k} \in \Lambda \cap I_j$ satisfying $|\eta_k - \eta_l|\geq K$ for $k
\neq l$ and summing over all indices $j$, one gets
\[
 \|F\|^p_{L^p(J)} \leq C^p\left[ \|F\|^p_{L^r_p(J)} + \sum_{n
\in \mathbb{Z}} (\lambda_{n+1}-\lambda_{n-1})|F(\lambda_n)|^p\right] \leq 
C^p\left[ \|F\|^p_{L^r_p(J)} + \|f\|^p_{\text{simp},W}\right],
\]
where $C=C(r,K)$ and $f = F|_\lambda$.

Now one can repeat the arguments of the proof of (\ref{trr-a}) in
Theorem~\ref{theorem-relation}, replacing $\|f\|_{\text{eq},W}$ by
$\|f\|_{\text{simp},W}$ and using this last inequality instead of
\eqref{eq:theorem-relation-*}. Here, one must use the inequalities
$\|f\|_{\text{eq},L} \leq \|f\|_{\text{simp},W}$, which is trivial, and
$\|f\|_{W^r_p(I)|_\Lambda} \geq C' \|f\|_{\text{simp},W}$, which is proved in
the same way as (\ref{norm-lower}) in Proposition~\ref{lower}. Then, it is easy
to see that all the steps made after \eqref{eq:theorem-relation-*} in the proof
of (\ref{trr-a}) in
Theorem~\ref{theorem-relation} remain true when we replace
$\|f\|_{\text{eq},W}$ with $\|f\|_{\text{simp},W}$.

\section{Open questions}
\label{open-questions}

The first open questions are whether one can generalize our results on
$\mathbb{R}$
to a higher number of derivatives. In particular, we ask the following.

\begin{question}
\label{quest-1}
Are \eqref{eq:equivalent-L} and \eqref{eq:equivalent-W} true for $r \geq 3$?
\end{question}

Theorem~\ref{theorem-relation} shows that if
\eqref{eq:equivalent-L} is true for some values of $r,p$, then
\eqref{eq:equivalent-W} is also true for the same values of $r,p$.

\begin{question}
Suppose that the answer to Question~\ref{quest-1} is affirmative. Can one give
explicit quasi-optimal spline interpolators for $r \geq 3$?
\end{question}

As before, Theorem~\ref{theorem-relation} shows that it is enough to prove that
some interpolator is quasi-optimal for $L^r_p$.

Another matter is whether one can give similar results in higher dimension,
i.e.~on $\mathbb{R}^d$, $d \geq 2$.  Fefferman, Israel and Luli show in
\cite{FeffermanIsraelLuliJune} that it is impossible to obtain bounded depth
interpolators for $d \geq 2$. Hence, one cannot obtain simple interpolators
like $\Phi_1$ and $\Phi_2$ for a general configuration of points.

However, it might be possible to obtain nice formulas and explicit interpolators
if one imposes some kind of regularity conditions on the configuration of
points.
The conditions should be mild enough to allow the usual cases which appear in
applications.

One possible condition is the minimum angle condition, which
is widely used in numerical analysis. It means that in the triangular mesh
formed by the points, the interior angles of every triangle are uniformly
bounded form below. See \cite{Brandts} for a review on this and other usual
geometric conditions in numerical analysis.

Hence, we ask the following question.

\begin{question}
 Can one give simple formulas for the trace space norm and quasi-optimal
explicit interpolators in $L^r_p(\mathbb{R}^d)$ and $W^r_p(\mathbb{R}^d)$ if one
imposes some regularity conditions on the allowed configurations of points?
\end{question}

\section*{Acknowledgments}

The author would like to thank Dmitry Yakubovich for posing this
problem and for all his helpful advice during the completion of this work.

\end{document}